\documentclass[a4paper,11pt]{amsart}
\usepackage{iftex}
\ifPDFTeX
\usepackage[utf8]{inputenc}
\fi

\usepackage{amsmath,amssymb,amsfonts,amsthm,bm}
\usepackage{graphicx}
\usepackage{enumitem}
\usepackage{geometry}
\usepackage{color}

\usepackage[final]{microtype}
\microtypesetup{nopatch=footnote}
\usepackage{xurl}
\usepackage{hyperref}

\newcommand*{\abs}[1]{\lvert#1\rvert}
\newcommand{\res}{\!\upharpoonright\!}

\DeclareMathOperator{\dom}{dom}
\DeclareMathOperator{\ran}{ran}

\newtheorem{theorem}{Theorem}[section]
\newtheorem{lemma}[theorem]{Lemma}

\newtheorem{corollary}[theorem]{Corollary}
\newtheorem{question}[theorem]{Question}

\theoremstyle{definition}
\newtheorem{definition}[theorem]{Definition}

\begin{document}
\title{Robinson Splitting Theorem and \(\Sigma_1\) Induction}

\author[Liu]{Yong Liu}
\address[Liu]{School of Information Engineering\\
Nanjing Xiaozhuang University\\
CHINA}
\email{\href{mailto:liuyong@njxzc.edu.cn}{liuyong@njxzc.edu.cn}}

\author[Peng]{Cheng Peng}
\address[Peng]{Institute of Mathematics\\
Hebei University of Technology\\
CHINA}
\email{\href{mailto:pengcheng@hebut.edu.cn}{pengcheng@hebut.edu.cn}}

\author[Sun]{Mengzhou Sun}
\address[Sun]{Institute of Mathematics, University of Warsaw, Poland}
\email{\href{mailto:m.sun3@uw.edu.pl}{m.sun3@uw.edu.pl}}



\thanks{Liu's research was partially funded by Nanjing Xiaozhuang University (No.~2022NXY39). Peng's research was partially funded by the National Natural Science Foundation of China (No.~12271264).}

\begin{abstract}
    The Robinson Splitting Theorem states that a c.e.\ degree \(\bm{b}\) splits over any low c.e.\ degree \(\bm{c}<\bm{b}\).
    We prove that a weaker version of this theorem holds in models of \(\mathsf{P}^-+\mathsf{I}\Sigma_1\), with lowness replaced by \emph{superlowness}.
\end{abstract}

\maketitle

\section{Introduction}
Friedberg--Muchnik Theorem~\cite{RN145,RN146} and the Sacks Splitting Theorem~\cite{RN158} are fundamental results in recursion theory, proved by finite-injury priority arguments.
From the viewpoint of injury bounds, the Sacks Splitting Theorem differs from the Friedberg--Muchnik Theorem.
If the number of injuries can be bounded effectively, the argument is called a bounded-type finite-injury argument.
In this sense, the Friedberg--Muchnik construction is bounded type, while the Sacks Splitting construction is unbounded type.
This distinction is also confirmed in reverse recursion theory.
Chong and Mourad~\cite{RN97} proved that the Friedberg--Muchnik Theorem is provable in models of \(\mathsf{P}^-+\mathsf{I}\Sigma_0+\mathsf{B}\Sigma_1+\mathrm{Exp}\).
By contrast, the Sacks Splitting Theorem is equivalent to \(\mathsf{I}\Sigma_1\) over the base theory \(\mathsf{P}^-+\mathsf{I}\Sigma_0+\mathsf{B}\Sigma_1+\mathrm{Exp}\)~\cite{RN82,RN105}.
We remark that it is not the case that bounded-type finite-injury arguments can always be carried out in models of \(\mathsf{P}^-+\mathsf{I}\Sigma_0+\mathsf{B}\Sigma_1+\mathrm{Exp}\); Chong and Yang~\cite{RN173} showed that the existence of a noncomputable low c.e.\ degree is equivalent to \(\mathsf{I}\Sigma_1\) over \(\mathsf{P}^-+\mathsf{I}\Sigma_0+\mathsf{B}\Sigma_1+\mathrm{Exp}\).

In~\cite{RN82}, Mytilinaios exploited a blocking technique from \(\alpha\)-recursion theory introduced by Shore~\cite{RN382}, and proved:
\begin{theorem}[Sacks Splitting Theorem~\cite{RN158,RN82}]\label{thm:sacks}
    The following holds in models of \(\mathsf{P}^-+\mathsf{I}\Sigma_1\):
    Let \(\bm{b}\) be a c.e.\ degree.
    Then there exist incomparable c.e.\ degrees \(\bm{a_0}\) and \(\bm{a_1}\) such that \(\bm{b}= \bm{a_0}\vee \bm{a_1}\).
\end{theorem}

Robinson proved the following refinement of Sacks Splitting Theorem:
\begin{theorem}[Robinson's Splitting Theorem~\cite{RN129}]\label{thm:robinson}
    Let \(\bm{b}>\bm{c}\) be c.e.\ degrees with \(\bm{c}\) a low c.e.\ degree.
    Then there exist incomparable c.e.\ degrees \(\bm{a_0}\) and \(\bm{a_1}\) such that \(\bm{b}= \bm{a_0}\vee \bm{a_1}\) with \(\bm{a_i}>\bm{c}\) for \(i<2\).
\end{theorem}
In the construction, Robinson introduced a method, now often called Robinson's trick, for approximating an answer computable from \(\bm{0}'\).
This method allows only finitely many wrong guesses, which makes the construction possible.
(Lachlan's Nonsplitting Theorem~\cite{RN130} asserts that there exists a c.e.\ degree that cannot split over any strictly lesser c.e.\ degree.)
This introduces another finiteness issue in addition to the finite-injury behavior in the Sacks Splitting Theorem.
Can this additional finiteness be treated in models of \(\mathsf{I}\Sigma_1\)?
Although it is often believed that finite-injury priority arguments can be carried out in models of \(\mathsf{I}\Sigma_1\), this is not immediate for the Robinson Splitting Theorem.
We prove a weaker version of Robinson's theorem in models of \(\mathsf{I}\Sigma_1\), replacing the lowness of \(\bm{c}\) with \emph{superlowness} (see Section~\ref{sec:discussion} for more discussion).
Recall that a set \(C\) is \emph{low} if \(C'\le_T \varnothing'\), and \emph{superlow} if \(C'\le_{wtt} \varnothing'\).
In models of \(\mathsf{P}^-+\mathsf{I}\Sigma_1\), we also have \(A\le_{wtt}\varnothing'\) if and only if \(A\) is \(\omega\)-c.e.\@. 
Recall that \(A\) is \(\omega\)-c.e.\@ if there exist computable function \(f\) and \(g\) such that \(A(x)=\lim_s f(x,s)\), \(f(x,0)=0\), and \(|\{s\mid f(x,s+1)\neq f(x,s)\}|< g(x)\).
Hence, if \(C\) is superlow, then \(C'\) is \(\omega\)-c.e.\@, which further implies that \(C\) is hyperregular (see Lemma~\ref{lem:superlow_hyperregular}).
The properties of being \(\omega\)-c.e.\ and hyperregular are two key ingredients in our proof (especially for Lemma~\ref{lem:Robinson_block_of_requirements}).

Our main result is the following:
\begin{theorem}\label{thm:main}
    The following holds in models of \(\mathsf{P}^-+\mathsf{I}\Sigma_1\):
    Given c.e.\ degrees \(\bm{b},\bm{c},\bm{d}\) with \(\bm{c}\) superlow and \(\bm{d}\nleq \bm{c}\), there exist incomparable c.e.\ degrees \(\bm{a_0},\bm{a_1}\)
    such that \(\bm{b}= \bm{a_0}\vee \bm{a_1}\) with \(\bm{d}\nleq \bm{a_i}\vee \bm{c}\) for \(i<2\).
\end{theorem}

Taking \(\bm{d}=\bm{b}>\bm{c}\) recovers Robinson's Splitting Theorem.

\begin{corollary}
    The following holds in models of \(\mathsf{P}^-+\mathsf{I}\Sigma_1\):
    Let \(\bm{b}>\bm{c}\) be c.e.\ degrees with \(\bm{c}\) superlow.
    Then there exist incomparable c.e.\ degrees \(\bm{a_0},\bm{a_1}\) such that \(\bm{b}= \bm{a_0}\vee \bm{a_1}\) with \(\bm{a_i}>\bm{c}\) for \(i<2\).
\end{corollary}

In Section~\ref{sec:preliminary}, we recall the basic setting of reverse recursion theory and the inductive hierarchy;
for further background, see Chong, Li, and Yang~\cite{RN100}.
In Section~\ref{sec:sacks}, we present a proof of the Sacks Splitting Theorem in models of \(\mathsf{P}^-+\mathsf{I}\Sigma_1\).
The reader may also consult the original paper of Mytilinaios~\cite{RN82}; our presentation differs slightly from the original one.
Building on this framework, we make the necessary modifications in Section~\ref{sec:robinson} to prove our main result, Theorem~\ref{thm:main}.
More precisely, the ideas of blocking and dynamic priority assignments remain the same.
The main difficulty lies in the proof of Lemma~\ref{lem:Robinson_block_of_requirements}, which is the analogue of Lemma~\ref{lem:block_of_requirements}.

\section{Preliminaries}\label{sec:preliminary}
We work over the base theory \(\mathsf{P}^-+\mathsf{I}\Sigma_1\).
Here \(\mathsf{P}^-\) denotes the axioms of \(\mathsf{PA}\) with the induction scheme omitted.
Also, \(\mathsf{I}\Sigma_n\) (resp. \(\mathsf{I}\Pi_n\)) denotes the induction scheme for \(\Sigma_n\) (resp. \(\Pi_n\)) formulas, and \(\mathsf{B}\Sigma_n\) (resp. \(\mathsf{B}\Pi_n\)) denotes the collection scheme for \(\Sigma_n\) (resp. \(\Pi_n\)) formulas.
\(\mathsf{L}\Sigma_1\) (resp. \(\mathsf{L}\Pi_1\)) denote that every nonempty \(\Sigma_1\) (resp. \(\Pi_1\)) subset has a least element.
Finally, \(\mathrm{Exp}\) denotes the axiom \(\forall x \exists y\,(y=\exp(x))\), where \(\exp:x\mapsto 2^x\) is the exponential function.

We adopt the standard pairing function, where the code of the pair \((a,b)\) under this function is denoted by \(\langle a, b\rangle\).
Let \(M\) be a model of \(\mathsf{P}^-+\mathsf{I}\Sigma_1\).
For any \(c\in M\), we identify \(c\) with a subset of \(M\) by defining \(x\in c\) to mean that the \(x\)-th digit of the binary expansion of \(c\) is \(1\).
We say a subset \(A\) of \(M\) is \emph{\(M\)-finite} if it is coded by some \(c\in M\).
In particular, \(A\) is bounded in \(M\).

It is well known that:
\begin{theorem}\label{thm:inductive_hierarchy}
    (Paris and Kirby~\cite{RN319}) Work in \(\mathsf{P}^-+\mathsf{I}\Sigma_0\).
    For all \(n \geq 1\),
    \begin{enumerate}
        \item \(\mathsf{I}\Sigma_1\iff \mathsf{I}\Pi_1 \iff \mathsf{L}\Sigma_1 \iff \mathsf{L}\Pi_1\).
        
        \item \(\mathsf{B}\Sigma_{n+1} \Leftrightarrow \mathsf{B}\Pi_n\).
        
        \item \(\mathsf{I}\Sigma_{n+1} \Rightarrow \mathsf{B}\Sigma_{n+1} \Rightarrow \mathsf{I}\Sigma_n\), and the implications are strict.
    \end{enumerate}
    (Slaman~\cite{RN317}) Work in \(\mathsf{P}^-+\mathsf{I}\Sigma_0+\mathrm{Exp}\).
    For all \(n\ge 1\),
    \begin{enumerate}[resume]
        \item \(\mathsf{B}\Sigma_n\Leftrightarrow \mathsf{I}\Delta_n\).
    \end{enumerate}
\end{theorem}
In Section~\ref{sec:discussion}, we will see that \(\mathsf{B}\Sigma_2\) proves the original Robinson Splitting Theorem.
Since \(\mathsf{B}\Sigma_2\) is strictly stronger than \(\mathsf{I}\Sigma_1\), the proof of Theorem~\ref{thm:main} requires additional care.
Although we do not assume \(\mathsf{B}\Sigma_2\), the following lemmas are still available and make the proof possible.

The first lemma is a well-known result.
\begin{lemma}[H. Friedman]\label{lem:Friedman}
    For every model \(M\) of \(\mathsf{P}^-+\mathsf{I}\Sigma_1\), every partial computable function maps bounded subsets of \(M\) to bounded subsets of \(M\).
\end{lemma}
We will also need an analogue of Lemma~\ref{lem:Friedman} for oracle computations.
\begin{definition}\label{def:hyperregular}
    A set \(A\) is \emph{hyperregular} if, for every model \(M\) of \(\mathsf{P}^{-} + \mathsf{I}\Sigma_0 + \mathrm{Exp}\),
    every partial \(A\)-computable function maps a bounded subset of \(M\) to a bounded subset of \(M\).
\end{definition}

Hyperregular sets can be characterized in several ways.
\begin{lemma}[\cite{RN384}]\label{lem:hyperregular}
    Let \(M\) be a model of \(\mathsf{P}^- + \mathsf{I}\Sigma_0 + \mathrm{Exp}\), and let \(A \subseteq M\) be regular.
    Then the following are equivalent:
    \begin{enumerate}
        \item \(A\) is hyperregular.
        \item \(\mathsf{I}\Sigma_1^A\) holds; that is, \(\mathsf{I}\Sigma_1\) relative to \(A\) holds.
        \item For every \(e\), \(W_e^A\) is regular.
    \end{enumerate}
\end{lemma}
Here, \(A\subseteq M\) is \emph{regular} if, for each \(n\in M\), the restriction \(A \res n := \{x\in A\mid x<n\}\) is \(M\)-finite.
It is known that in models of \(\mathsf{P}^- + \mathsf{I}\Sigma_1\), every c.e.\ set is regular.
In fact, this property also holds for \(\omega\)-c.e.\ sets:

\begin{lemma}\label{lem:omega_c_e_regular}
    Let \(M\) be a model of \(\mathsf{P}^-+\mathsf{I}\Sigma_1\) and \(A\subseteq M\) be \(\omega\)-c.e.\@.
    Then \(A\) is regular.
\end{lemma}
\begin{proof}
    Suppose \(A\) is \(\omega\)-c.e.\@. Let \(f\) and \(b\) be the computable functions such that \[
        A(x)=\lim_{s\to \infty} f(x,s)
    \]
    where for each \(x\in M\), \(f(x,0)=0\), and \(|\{s\mid f(x,s+1)\neq f(x,s)\}|< b(x)\).

    Let \(C\) be the change set for \(f\), i.e., we enumerate \(\langle x,i-1\rangle\) into \(C\) at stage \(s\) if \(f(x,s+1)\neq f(x,s)\) for the \(i\)-th time.
    So \(C\) is a c.e.\ set and thus regular.
    
    By the assumption of \(A\), \(\langle x,i\rangle \in C\) implies \(i<b(x)\).
    For any \(n\in M\), let \(d=\max\{\, b(x)\mid x<n\}\), whose existence follows from \(\mathsf{I}\Sigma_1\).
    Let \(t\) code \(C\) up to \(\langle n, d\rangle\).
    Now, \(x\in A\) if and only if \(|\{i\mid \langle x,i\rangle \in C\}|\) is odd.
    So for any \(x<n\),
    \[
        x\in A\iff \exists i<d \left( \langle x,i \rangle\in t \land \langle x,i+1\rangle \notin t \land i \text{ is odd}\right)
    \]
    Hence \(A\res n\) is coded in \(M\) for any \(n\in M\), so \(A\) is regular.
\end{proof}

\begin{lemma}\label{lem:superlow_hyperregular}
    If \(C\) is a superlow c.e.\ set, then \(C\) is hyperregular.
\end{lemma}
\begin{proof}
    Since \(C\) is superlow, \(C'\) is \(\omega\)-c.e.\ and therefore regular by Lemma~\ref{lem:omega_c_e_regular}.
    For each \(e\), \(W_e^C\le_1 C'\) implies that \(W_e^C\) is regular. 
    By Lemma~\ref{lem:hyperregular}, \(C\) is hyperregular.
\end{proof}

It was shown in~\cite{RN365,RN320} that, assuming \(\mathsf{B}\Sigma_2\), every incomplete c.e.\ set is hyperregular.
Thus, this issue only arises in models of \(\mathsf{I}\Sigma_1\).

Finally, we introduce some notation:
If \(\sigma\) is an initial segment of a set \(X\), we write \(X\in [\sigma]\).
If \(W\) is a set of finite strings, we write \(X\in [W]\) if \(X\in [\sigma]\) for some \(\sigma\in W\).

\section{Proof of Theorem~\ref{thm:sacks}}\label{sec:sacks}
We will prove a slightly stronger statement.
Let \(M\) be a model of \(\mathsf{P}^-+\mathsf{I}\Sigma_1\).
Given \(B,D\subseteq M\) c.e.\ sets with \(D\) noncomputable, we exhibit c.e.\ sets \(A_0,A_1\subseteq M\) such that \(B=A_0\sqcup A_1\) and \(D\nleq_T A_i\) for \(i<2\).
Taking \(D=B\) gives the original theorem.
Taking \(D<_T B\) yields other interesting corollaries.

We have a global requirement that \(B=A_0\sqcup A_1\), which implies \(B\equiv_T A_0\oplus A_1\).
The other requirements are the following:
\begin{align*}
    P_e: D&=\Phi_e^{A_0}\to D=\Delta_e;\\
    Q_e: D&=\Psi_e^{A_1}\to D=\Gamma_e.
\end{align*}
where \(\Delta_e\) and \(\Gamma_e\) are local functionals that will be built during the construction.
Given that \(D\) is noncomputable, once the requirements are satisfied for all \(e\in M\), we are done.
We remark that the injury set \(I_e=\{s\mid P_e\text{ is injured at }s\}\) is \(\Sigma_1\).
However, the statement that \(I_e\) is finite is \(\Sigma_2\).
The following induction only works over \(\mathsf{I}\Sigma_2\):
\[
    (\forall k<e, I_k\text{ is finite}) \to I_e\text{ is finite}
\]
To resolve this, as in Mytilinaios' paper~\cite{RN382}, there are two ideas:
the first is the blocking method, and the second is to use dynamic priority assignments.
We first review the basic strategy of Sacks Splitting Theorem.

\subsection{A single requirement}\label{sec:SST_single_requirement}
Without loss of generality, we take \(P_e\)-requirements as an example.
First we introduce some standard notation.
At stage~\(s\), if the \(P_e\)-requirement is eligible to act, then \(s\) is called a \(P_e\)-stage.
Let \[
    \ell_s(e)=\max \{ y\mid \forall x\le y, \Phi^{A_{0,s}}_{e,s}(x)\downarrow = D_s(x)\}.
\]
A \(P_e\)-stage is a \emph{\(P_e\)-expansionary stage} if \(\ell_s(e)>\ell_t(e)\) for all \(P_e\)-stages \(t<s\).

The basic idea to satisfy this requirement is the following:
at a \(P_e\)-expansionary stage \(s\), for those \(x\le \ell_s(e)\) where \(\Delta_e(x)\) is not yet defined,
we define \(\Delta_e(x)=D_s(x)\) and call \((\sigma,x,k)\) the \emph{diagonalizing triple},
where \(\sigma=A_{0,s}\res s\) and \(k=\Delta_e(x)\).
Note that \(\Phi_e^\sigma(x)=D_s(x)=\Delta_e(x)=k\).
We also put a restraint on \(A_0\res s\) to prevent points from being enumerated into \(A_0\res s\).

There are two cases for \(\Delta_e(x)\) with diagonalizing triple \((\sigma,x,k)\):
the first is \(\Delta_e(x)=D(x)\).
The second is \(\Delta_e(x)\neq D(x)\).
In the latter case, there exists some stage \(t>s\) such that \(D_t(x)\neq \Delta_e(x) = k\).
By successful restraint on \(A_0\res s\), we have \(A_0\in [\sigma]\) and hence \(\Phi_{e}^{A_0}(x)=k\neq D(x)\).
As long as the restraint is successful, we have \(\Phi_e^{A_0}(x)\neq D(x)\) and we claim the \(P_e\)-requirement \emph{diagonalizes at \(x\)}.

Returning to the satisfaction of the \(P_e\)-requirement.
\begin{enumerate}
    \item If \(\Delta_e(x)\neq D(x)\) for some \(x\in \dom \Delta_e\), then \(P_e\) diagonalizes at \(x\) and therefore is satisfied.
    \item If \(\Delta_e(x) = D(x)\) for all \(x\in \dom \Delta_e\) and \(\dom \Delta_e\neq M\), then there must be only boundedly many \(P_e\)-expansionary stages.
    Therefore, \(P_e\) is satisfied.
    \item If \(\Delta_e(x) = D(x)\) for all \(x\in \dom \Delta_e\) and \(\dom \Delta_e= M\),
    this case does not happen as \(D=\Delta_e\) contradicts the assumption that \(D\) is noncomputable.
\end{enumerate}

In the first two cases, \(\dom \Delta_e\) is bounded, and therefore the restraint put on \(A_0\) is bounded.

We summarize the discussion as the following lemma:
\begin{lemma}\label{lem:single_requirement}
    Suppose there exists some \(s_0\) after which \(P_e\) (\(Q_e\) respectively) is not initialized.
    Then \(P_e\) (\(Q_e\) respectively) is satisfied.
    Moreover, the restraint put on \(A_0\) (\(A_1\) respectively) is bounded.
\end{lemma}
\begin{proof}
    Take the \(P_e\)-requirement as an example.
    By the discussion above, \(\dom \Delta_e\) is bounded and therefore the restraint put on \(A_0\) is bounded.
\end{proof}

\subsection{A block of requirements}\label{sec:SST_blocking}
To carry out the proof in an \(\mathsf{I}\Sigma_1\) model, Mytilinaios~\cite{RN382} adapted the blocking method from \(\alpha\)-recursion theory.
We review this technique here.
The set \(\Lambda=\{P_{e_0},P_{e_0+1},\cdots, P_{e_1}\}\) is referred to as a \emph{block} of \(P\)-requirements (\(Q\)-requirements would be the same).
For notational convenience, we also confuse the set \(\Lambda\) with the set of their indices \(\{e_0,e_0+1,\cdots, e_1\}\).
So \(P_e\in \Lambda\) and \(e\in \Lambda\) are both valid.

As \(P\)-requirements put restraint only on \(A_0\), there are no conflicts between requirements in \(\Lambda\).
We let each requirement in \(\Lambda\) follow exactly the same strategy as in the single-requirement case.
As pointed out in Lemma~\ref{lem:single_requirement}, the restraint put on \(A_0\) is bounded for each requirement in \(\Lambda\).
But the overall restraint put on \(A_0\) is not obviously bounded in a model of \(\mathsf{I}\Sigma_1\).
We prove that it is in fact bounded.

Note that whenever \(P_e\) diagonalizes at some \(x\), it will not have another expansionary stage.
Let \(f:\Lambda\to M\) be defined by \(f(P_e)=s\) if \(P_e\) diagonalizes at stage \(s\).
Since \(f\) is partial computable with bounded domain, \(\ran f\) is bounded by some \(s_0\) by Lemma~\ref{lem:Friedman}.
Then after stage \(s_0\), no requirement in \(\Lambda\) ever diagonalizes.
Moreover, if \(\Delta_e(x)=D_s(x)\) is defined at stage \(s>s_0\) for some \(P_e\in \Lambda\) and \(x\in M\), then it is permanently correct;
otherwise, \(P_e\) diagonalizes at some stage \(t>s\), contradicting the choice of \(s_0\).

Suppose, towards a contradiction, that \(\bigcup_{e\in \Lambda}\dom \Delta_e = M\).
We claim that \(D\) is computable.
To see this, given \(x>s_0\), search for an \(e\in \Lambda\) and a stage \(s>s_0\) such that \(\Delta_e(x)\) is defined at stage \(s\).
Then \(D(x)=D_s(x)=\Delta_e(x)\).
Therefore \(D\) is computable.
Contradiction.
Hence, \(\bigcup_{e\in \Lambda}\dom \Delta_e\) is bounded by some \(b\).

Let \(g:\Lambda\times b\to M\) be defined by \(g(e,x)=s\) if \(\Delta_e(x)\) is defined at stage \(s\).
Since \(g\) is partial computable with a bounded domain, \(\ran g\) is bounded by some \(s_1>s_0\) by Lemma~\ref{lem:Friedman}.
That is, after stage \(s_1\), no requirement in \(\Lambda\) ever has another expansionary stage.
Therefore, the restraint put by requirements in \(\Lambda\) on \(A_0\) is bounded by \(s_1\).

To summarize, we have the following lemma:
\begin{lemma}\label{lem:block_of_requirements}
    Suppose there exists some \(s_0\) after which requirements in \(\Lambda\) are not initialized.
    Then all requirements in \(\Lambda\) are satisfied and there exists some stage \(s_1>s_0\)
    after which no requirement in \(\Lambda\) ever diagonalizes or has an expansionary stage.
    Also, the restraint put on \(A_0\) is bounded by \(s_1\).
\end{lemma}

We remark that the Robinson case (Lemma~\ref{lem:Robinson_block_of_requirements}) is substantially more subtle than Lemma~\ref{lem:block_of_requirements}, and its proof requires the full strength of both the hyperregularity of \(C\) and the fact that \(C'\) is \(\omega\)-c.e.
These properties are essential for bounding the overall restraint.

\subsection{Dynamic priority assignments}\label{sec:SST_dynamic_assignments}
We have blocks \(\Lambda(i)\) for \(P\)-requirements and \(\Upsilon(i)\) for \(Q\)-requirements, where \(i\) is the index of the block.
We use \(\Lambda_s(i)\) and \(\Upsilon_s(i)\) to denote the set of requirements contained in the block at stage \(s\),
and sometimes omit the subscript \(s\) when there is no confusion.
They are determined by the \emph{dynamic priority assignments} \(\lambda_s:\{P_e\mid e\in M\}\to M\) and \(\mu_s:\{Q_e\mid e\in M\}\to M\) by
\[
    \Lambda_s(i)=\lambda_s^{-1}(i) \quad \text{and} \quad \Upsilon_s(i)=\mu_s^{-1}(i).
\]
Requirements in the same block have the same priority order.
The blocks are ordered by priority as \(\Lambda(0)<\Upsilon(0)<\Lambda(1)<\Upsilon(1)<\cdots\).

\(P_m\), where \(m=\max \Lambda_{s}(i)\), is the \emph{tail requirement} in the block \(\Lambda(i)\) at stage \(s\);
\(Q_n\), where \(n=\max \Upsilon_{s}(i)\), is the \emph{tail requirement} in the block \(\Upsilon(i)\) at stage \(s\).
The idea is that if a block, say \(\Lambda(i)\), is initialized at stage \(s\),
we will add more \(P\)-requirements of lower priority into this block to compensate for the times that \(\Lambda(i)\) is initialized.

The dynamic priority assignments are defined as follows.
At the end of stage \(s\), suppose no block is initialized during the stage; we let \(\lambda_s(P_e)=\lambda_{s-1}(P_e)\) and \(\mu_s(Q_e)=\mu_{s-1}(Q_e)\) for all \(e\in M\).
Otherwise, let \(\Lambda(i)\) or \(\Upsilon(i)\) be the block of requirements of highest priority order that is initialized during the stage.
If it is \(\Lambda(i)\), let \(P_m\) be the tail requirement in \(\Lambda(i)\):
\begin{enumerate}
    \item for \(j\le m\), \(\lambda_s(P_j)=\lambda_{s-1}(P_j)\);
    \item for \(m<j\le s\), \(\lambda_s(P_j)=i\);
    \item for \(s<j\), \(\lambda_s(P_j)=\lambda_{s-1}(P_j)-(\lambda_{s-1}(P_s)-i)\).
    \item for \(j\in M\), \(\mu_s(Q_j)=\mu_{s-1}(Q_j)\).
\end{enumerate}
Item~(3) is equivalent to saying that \(\lambda_s(P_{s+j})=i+j\) for \(j\in M\); the form used in Item~(3) reveals the nonincreasing nature of \(\lambda_s\) more directly.

Analogously, if it is \(\Upsilon(i)\), let \(Q_n\) be the tail requirement in \(\Upsilon(i)\):
\begin{enumerate}
    \item for \(j\le n\), \(\mu_s(Q_j)=\mu_{s-1}(Q_j)\);
    \item for \(n<j\le s\), \(\mu_s(Q_j)=i\);
    \item for \(s<j\), \(\mu_s(Q_j)=\mu_{s-1}(Q_j)-(\mu_{s-1}(Q_s)-i)\).
    \item for \(j\in M\), \(\lambda_s(P_j)=\lambda_{s-1}(P_j)\).
\end{enumerate}
Item~(3) is equivalent to saying that \(\mu_s(Q_{s+j})=i+j\) for \(j\in M\).

From the definition of \(\lambda_s\) and \(\mu_s\), we have the following lemma:
\begin{lemma}\label{lem:nonincreasing}
    For each \(e,s\in M\) with \(s\ge 1\), \(\lambda_{s}(P_e)\le \lambda_{s-1}(P_e)\) and \(\mu_{s}(Q_e)\le \mu_{s-1}(Q_e)\).
\end{lemma}

\begin{lemma}\label{lem:final_block_order_exists}
    For each \(e\in M\), \(\lim_s \lambda_s(P_e)\) and \(\lim_s \mu_s(Q_e)\) exist.
\end{lemma}
\begin{proof}
    The set \(\{i\mid \exists s ( \lambda_s(P_e)=i)\}\) is a \(\Sigma_1\) set and so it has a least element \(i_0\) by \(\mathsf{L}\Sigma_1\).
    Since \(\lambda_{s+1}(P_e)\le \lambda_s(P_e)\), we have \(\lim_s \lambda_s(P_e)=i_0\).
    Similarly, \(\lim_s \mu_s(Q_e)\) exists.
\end{proof}

Take \(P\)-requirements as an example.
We define \(i=\lim_s \lambda_s(P_e)\) to be the \emph{final priority order} of \(P_e\).
If there is a least stage \(s\) such that for all \(t>s\), \(\Lambda_t(i)=\Lambda_s(i)\), we say that the block \(\Lambda(i)\) \emph{stabilizes} at stage \(s\).
This stage \(s\) agrees with the stage when \(\Lambda(i)\) is initialized for the last time, according to the definition of dynamic priority assignments.

\subsection{The construction and verification}\label{sec:SST_construction}
Let us now introduce some more terminology and present the construction.
At stage \(s\), the restraint set by \(\Lambda_s(i)\) is denoted by \(\rho_s(i)\) and the restraint set by \(\Upsilon_s(i)\) is denoted by \(\tau_s(i)\).
If no restraint is set up, we let \(\rho_s(i)=\tau_s(i)=-1\).
For a number \(x\in M\), we say \(x\) \emph{threatens} \(\Lambda_s(i)\) or \(\Upsilon_s(i)\) if \(x\le \rho_s(i)\) or \(x\le \tau_s(i)\), respectively.
When we initialize a block, we initialize all requirements in the block, and initialize all blocks of lower priority.
If a requirement is initialized, its \emph{local} functional \(\Delta_e\) or \(\Gamma_e\) and all diagonalizing triples are canceled.

We assume that for each \(s\in M\), \(B_s\setminus B_{s-1}\) is a singleton set if \(s\) is odd; an empty set if \(s\) is even.

\textbf{The construction:}
At stage~\(s\), if \(s\) is odd, we start with Part I; if \(s\) is even, we start with Part II.

\textbf{Part I. Enumeration of \(A_0\) and \(A_1\):}
Let \(x\in B_s\setminus B_{s-1}\).
Let the block of requirements of least priority order that \(x\) threatens be \(\Lambda(i)\) (or \(\Upsilon(i)\), respectively), if it exists.
\begin{enumerate}[label=(i.\arabic*)]
    \item If it is \(\Lambda(i)\), then we enumerate \(x\) into \(A_1\).
    Initialize the block \(\Upsilon(i)\);
    \item If it is \(\Upsilon(i)\), then we enumerate \(x\) into \(A_0\).
    Initialize the block \(\Lambda(i+1)\).
\end{enumerate}
If no block of requirements is threatened, we enumerate \(x\) into \(A_0\).
We then proceed to Part III.

\textbf{Part II. Block strategy:}
Let the block \(\Lambda(0)\) be eligible to act.
Suppose, without loss of generality, that the block \(\Lambda(i)\) is eligible to act.
For each \(P_e\in \Lambda(i)\), we perform the following \(P_e\)-strategy:
\begin{enumerate}[label=(ii.\arabic*)]
    \item\label{it:P_success} Suppose it has diagonalized at some \(x\); we stop the \(P_e\)-strategy.
    \item\label{it:P_diagonalize} Suppose for some \(x\in \dom \Delta_e\) with \((\sigma,x,k)\) being the diagonalizing triple,
    \(\Delta_e(x)=k\neq D_s(x)\), we say that \(P_e\) diagonalizes at \(x\).
    Then we stop \(P_e\)-strategy.
    \item\label{it:P_expansionary} Suppose \(s\) is a \(P_e\)-expansionary stage; we define \(\Delta_e(x)=D_s(x)\) for all \(x\le \ell_s(e)\).
    We define \(\rho_s(i)=s\).
    Then we stop the \(P_e\)-strategy.
    \item\label{it:P_no_action} Suppose \(s\) is not a \(P_e\)-expansionary stage; we stop the \(P_e\)-strategy.
\end{enumerate}
If Item~\ref{it:P_expansionary} does not happen for all \(P_e\in \Lambda(i)\),
we let \(\rho_s(i)=\rho_{t}(i)\) where \(t<s\) is the last stage when \(\Lambda(i)\) is eligible to act.
We say that the \(P_e\)-requirement \emph{acts} at stage~\(s\) if Item~\ref{it:P_diagonalize} or Item~\ref{it:P_expansionary} happens.

After all requirements in \(\Lambda(i)\) are stopped, if some requirement in \(\Lambda(i)\) acts at stage \(s\), we initialize blocks of lower priority and start Part III;
otherwise, we let the block with the next highest priority order be eligible to act,
until we reach a block that contains \(P_s\) requirement or \(Q_s\) requirement,
at which point we start Part III.

\textbf{Part III. Dynamic priority assignments:}
We refer the reader to the rules for \(\lambda_s\) and \(\mu_s\) in Section~\ref{sec:SST_dynamic_assignments}.

This completes the construction.
We now proceed to the verification,
using Lemma~\ref{lem:single_requirement}, Lemma~\ref{lem:block_of_requirements}, and Lemma~\ref{lem:final_block_order_exists},
which have already been proved.

\begin{lemma}\label{lem:block_not_initialized}
    Suppose \(\Lambda(i)\) (or \(\Upsilon(i)\), respectively) is not initialized after \(s_0\).
    Then there exists some stage \(s_1 > s_0\) after which \(\Upsilon(i)\) (or \(\Lambda(i+1)\), respectively) is not initialized.
\end{lemma}
\begin{proof}
    By Lemma~\ref{lem:block_of_requirements}, there exists some \(s^* > s_0\) such that the requirements in \(\Lambda(i)\) do not act after \(s^*\),
    and the restraint put on \(A_0\) is bounded by \(s^*\).
    There exists some stage \(s_1 > s^*\) such that no point \(x \leq s^*\) is enumerated into \(B\).
    Hence, by Part I of the construction, \(\Lambda(i)\) is not threatened after \(s_1\).
    Thus, \(\Upsilon(i)\) is not initialized after \(s_1\).
\end{proof}

\begin{lemma}\label{lem:SST_satisfied}
    For each \(e \in M\), \(P_e\) and \(Q_e\) are satisfied.
\end{lemma}
\begin{proof}
    The case for \(Q_e\) is analogous; we take \(P_e\) as an example.
    By Lemma~\ref{lem:final_block_order_exists}, let \(i = \lim_s \lambda_s(P_e)\) be the final priority order of \(P_e\),
    and let \(s_0\) be the first stage when \(\lambda_s(P_e) = i\).
    For each \(s > s_0\), \(\Lambda(i-1)\) is not initialized at stage \(s\), since otherwise \(\lambda_s(P_e) \leq i-1\).
    By Lemma~\ref{lem:block_not_initialized}, there exists some stage \(s_1 > s_0\) after which \(\Upsilon(i-1)\) is not initialized.
    By the same lemma, there exists some stage \(s_2 > s_1\) after which \(\Lambda(i)\) is not initialized.
    By Lemma~\ref{lem:block_of_requirements}, \(P_e\) is satisfied.
\end{proof}

This completes the proof of Theorem~\ref{thm:sacks}.
Based on this construction, we prove Theorem~\ref{thm:main} in the next section.

We emphasize that the use of \(\mathsf{I}\Sigma_1\) is crucial in the proof of Lemma~\ref{lem:block_of_requirements} (see Section~\ref{sec:SST_blocking}),
particularly for establishing the boundedness of the restraint imposed on \(A_0\).
In this argument, it is essential that \(D\) is c.e., ensuring that the function \(f\) remains \(\Sigma_1\), so the reasoning works within the framework of \(\mathsf{I}\Sigma_1\).
If \(D\) were not c.e., \(f\) would not be \(\Sigma_1\), and the proof would break down in \(\mathsf{I}\Sigma_1\).
However, if we instead work in the standard model \(\mathbb{N}\), \(D\) could be any \(\Delta_2\) set, and our construction would require only minor modifications.

\section{Proof of Theorem~\ref{thm:main}}\label{sec:robinson}
We fix a model \(M\) of \(\mathsf{P}^- +\mathsf{I}\Sigma_1\).
Let \(B, C, D \subseteq M\) be c.e.\ sets with \(C\) superlow and \(D \nleq_T C\).
We construct \(A_0\) and \(A_1\) such that \(B = A_0 \sqcup A_1\), and these sets satisfy the following requirements for each \(e \in M\):
\begin{align*}
    P_e: &\quad D = \Phi_e^{A_0 \oplus C} \to D = \Delta_e^C;\\
    Q_e: &\quad D = \Psi_e^{A_1 \oplus C} \to D = \Gamma_e^C.
\end{align*}
Here, each \(P_e\)-requirement and \(Q_e\)-requirement builds and maintains local functionals \(\Delta_e^C\) and \(\Gamma_e^C\), respectively, toward contradictions.

Robinson's trick is to consider \[
    X = \{j \mid \exists \sigma \in W_j,\, C\in [\sigma] \}.
\]
\(X\) is \(\Sigma_1^C\) and therefore \[
    X\le_1 C'\le_{wtt} \varnothing'.
\]
Since \(C'\) is \(\omega\)-c.e.\@, so is \(X\).
Let \(p\) and \(q\) be computable functions such that \[
    X(j) = \lim_s p(j,s)
\]
for all \(j\in M\), \(p(j,0)=0\) and \[
    \abs{\{s\mid p(j,s+1)\neq p(j,s)\}} \le q(j).
\]
The idea for satisfying the \(P_e\)-requirement is as follows:
Suppose \(\Phi_{e,s}^{A_{0,s}\oplus C_s}(x)\downarrow = D_s(x)\).
We first check whether this convergent computation is \emph{certified}.
If it is, we define \(\Delta_e^{C_s}(x)=D_s(x)\); otherwise, we treat the computation as divergent and do nothing.
To certify a computation, we use a c.e.\ set \(W_j\), referred to as the \emph{guessing set}, whose index \(j\) is given by the recursion theorem.
This \(j\) depends on \(e\), \(x\), and the number of times a relevant requirement has been initialized or injured.
Next, we describe how we certify a computation---this is the original idea behind Robinson's trick.
Our treatment is somewhat more technical, as we must handle ``finiteness'' carefully in models of \(\mathsf{P}^- + \mathsf{I}\Sigma_1\);
in the standard model, the verification is more intuitive and is sometimes even omitted.

Suppose \(\Phi_{e,s}^{A_{0,s}\oplus C_s}(x)\downarrow = k\) with use \(l\), and let \(\theta=A_{0,s}\res l\) and \(\sigma = C_s\res l\).
We refer to \(\Phi_{e,s}^{A_{0,s}\oplus C_s}(x)\downarrow\) as a \emph{computation} and \((\theta, \sigma, x, k)\) as the \emph{axiom} of this computation.
We also have a set \(F_e(x)\), collecting axioms of certified computations.
The certification procedure for this computation is as follows:
\begin{enumerate}[label=(C.\arabic*),ref=(C.\arabic*)]
    \item If \(C_s \notin [W_j]\), enumerate \(\sigma\) into \(W_j\).
    \item\label{it:certify_check} Let \(t \ge s\) be the least stage such that either \(C_t \notin [\sigma]\) or \(p(j, t) = 1\).
    \begin{enumerate}[label=(C.\arabic{enumi}\alph*),ref=(C.\arabic{enumi}\alph*)]
        \item\label{it:certify_not} If \(C_t \notin [\sigma]\), this computation is \emph{not certified}. Skip Item~\ref{it:certify_true}.
        \item\label{it:certify_true} If \(p(j, t) = 1\), this computation is \emph{certified}.
        Enumerate \((\theta, \sigma, x, k)\) into \(F_e(x)\).
    \end{enumerate}
\end{enumerate}
The stage \(t\) in Item~\ref{it:certify_check} exists since
\[
    C\in [W_j] \iff \lim_s p(j,s) = 1.
\]

Suppose the computation is certified at stage \(s\).
We have for any \(u\) with \(s\le u\le t\), \(C_u\in [\sigma]\).
We also want \(A_{0,u}\in [\theta]\) to preserve the computation.
So we impose a restraint on \(A_0\res l\) to achieve this.
Conversely, if the computation is not certified at stage \(s\) (so \(C_t\notin[\sigma]\)),
then we place no restriction on \(A_0\res l\), allowing points to be enumerated into \(A_0\).
If for some \(u\) with \(s\le u\le t\), we observe \(\Phi_{e,u}^{A_{0,u}\oplus C_u}(x)\downarrow\),
the certification procedure will not certify this computation at stage \(u\), as we still have \(C_t \notin [\sigma]\).
So we may simply ignore this computation at stage \(u\).

For clarity of discussion, we say that a computation is \emph{injured} if \(A_0\res l\) changes while the computation is certified;
likewise, \(F_e(x)\) is said to be \emph{injured} if any of its certified computations become injured.
Once a computation is injured, we cannot certify any subsequent computations using the same guessing set,
as it may turn out that \(C\in[\sigma]\) so all computations are automatically certified.
In such cases, if we really want to have a meaningful certification, we must refresh the index \(j\) of the guessing set \(W_j\), and also clear the set \(F_e(x)\).
Eventually, we will show that \(F_e(x)\) will not be injured unless the \(P_e\)-requirement is initialized.

\(F_e(x)\) is \emph{initialized} if the \(P_e\)-requirement is initialized.
If \(F_e(x)\) is injured or initialized, we refresh the index \(j\) of the guessing set \(W_j\) and clear the set \(F_e(x)\).
Using the language of \(F_e(x)\), a computation is \emph{certified} at stage \(s\) for the first time if and only if its axiom is enumerated into \(F_e(x)\) at stage \(s\).

\begin{lemma}\label{lem:F_bounded}
    Suppose \(F_e(x)\) is not injured or initialized after \(s_0\).
    Then \(F_e(x)\) is bounded.
\end{lemma}
\begin{proof}
    Let \(W_j\) be the guessing set.
    Suppose \(\lim_s p(j,s) = 0\) with \(s_1\) being the stage such that for each \(s > s_1\), \(p(j,s) = 0\).
    Then for each \(s \geq s_1\), no computation can be certified.
    Therefore, for each \(s > s_1\), we enumerate no axioms into \(F_e(x)\).
    So, \(\abs{F_e(x)}\) is \(M\)-finite.

    Suppose \(\lim_s p(j,s) = 1\).
    Then \(C\in[W_j]\).
    Let \(\sigma\in W_j\) be such that \(C\in [\sigma]\).
    Suppose \(\sigma\) is enumerated into \(W_j\) at stage \(s_1\) with \((\theta,\sigma,x,k)\) the axiom of the computation
    \(\Phi_{e,s_1}^{A_{0,s_1}\oplus C_{s_1}}(x)\downarrow = k\).
    Since \(F_e(x)\) is not injured, \(A_{0,s}\in [\theta]\) for all \(s\ge s_1\).
    Since \(C\in [\sigma]\), this computation is certified at all stages \(s\ge s_1\).
    Therefore, \((\theta,\sigma,x,k)\) is the last axiom enumerated into \(F_e(x)\).
\end{proof}

\subsection{Single requirement}\label{sec:Robinson_single_requirement}
Without loss of generality, we take the \(P_e\)-requirement as an example.
We first consider the following subrequirement of \(P_e\), denoted \(P_e(x)\) for each \(x\in M\):
\[
    P_e(x): \quad D(x) = \Phi_e^{A_0\oplus C}(x) \to D(x) = \Delta_e^C(x).
\]

\textbf{\(P_e(x)\)-strategy:}
At stage \(s\), if \(F_e(x)\) is injured or initialized, we refresh the index \(j\) of the guessing set \(W_j\) and clear the set \(F_e(x)\).
If \(\Phi_e^{A_{0,s}\oplus C_s}(x)\uparrow\) or \(\Phi_e^{A_{0,s}\oplus C_s}(x)\downarrow \neq D_s(x)\), we do nothing;
otherwise, we proceed as follows:
\begin{enumerate}[label=(Px.\arabic*), ref=(Px.\arabic*)]
    \item\label{it:Px_defined} If \(\Delta_e^{C_s}(x)\downarrow\), we do nothing;
    \item\label{it:Px_not_defined} If \(\Delta_e^{C_s}(x)\uparrow\) and \(\Phi_{e,s}^{A_{0,s}\oplus C_s}(x)\downarrow = D_s(x)=k\) with \((\theta,\sigma,x,k)\) the axiom,
    we perform the certification procedure for this computation.
    \begin{enumerate}[label=(Px.\arabic{enumi}\alph*), ref=(Px.\arabic{enumi}\alph*)]
        \item\label{it:Px_not_certified} If the computation is not certified, we do nothing;
        \item\label{it:Px_certified} otherwise, we define \(\Delta_e^\sigma(x)=k\), call \((\theta,\sigma,x,k)\) the \emph{diagonalizing axiom},
        and put a restraint on \(A_0\res s\) to preserve the computation.
        (We assume that \(s>\abs{\theta}\).)
        We call \(P_e(x)\) \emph{acts} at stage \(s\).
        This diagonalizing axiom is automatically canceled at future stage \(t\) if \(C_t\notin [\sigma]\).
    \end{enumerate}
\end{enumerate}
Note that we do not perform the certification procedure in Item~\ref{it:Px_not_defined} if \(\Phi_{e,s}^{A_{0,s}\oplus C_s}(x)\downarrow \neq D_s(x)\).
It would be problematic to have a computation certified before comparing it with \(D_s(x)\),
because in such a case, we cannot place a restraint on \(A_0\), which opens the possibility that \(F_e(x)\) could be injured.
By first comparing and then performing the certification procedure, we avoid this problem and ensure that \(F_e(x)\) will not be injured.

Also note that we extend our restraint on \(A_0\) only when \(P_e(x)\) acts.
The restraint we place on \(A_0\) is probably stronger than necessary; moreover, we do not remove the restraint unless the \(P_e(x)\)-requirement is initialized.
This approach simplifies the presentation.

We justify Item~\ref{it:Px_defined}.
\begin{lemma}\label{lem:Px_lemma}
    Suppose \(F_e(x)\) is not injured or initialized after \(s_0\).
    \begin{enumerate}
        \item\label{it:Px_lemma_1} For each \(s>s_0\) with \(\Delta_e^{C_s}(x)\downarrow\), then \(\Phi_{e,s}^{A_{0,s}\oplus C_s}(x)\downarrow=D_s(x)\) if and only if
        \(\Delta_e^{C_s}(x)\downarrow=D_s(x)\).
        \item\label{it:Px_lemma_2} Suppose \(\Delta_e^C(x)\uparrow\), then either \(\Phi_e^{A_0\oplus C}(x)\uparrow\) or \(\Phi_e^{A_0\oplus C}(x)\downarrow \neq D(x)\).
    \end{enumerate}
\end{lemma}
\begin{proof}
    \textit{ (1):}
    \((\Rightarrow)\)
    Suppose towards a contradiction that \(\Delta_e^{\sigma}(x)\downarrow \neq D_s(x)\) with \(C_s\in [\sigma]\).

    Let \((\theta,\sigma,x,k)\) be the diagonalizing axiom claimed in Item~\ref{it:Px_certified} at a stage \(t\) with \(s_0<t<s\).
    Since \(A_0\res t\) is preserved and \(C_s\in [\sigma]\), we still have \(\Phi_{e,s}^{A_{0,s}\oplus C_s}(x)\downarrow = \Delta_e^{\sigma}(x)=k\).
    Since \(D_s(x)\neq k\), we have \(\Phi_{e,s}^{A_{0,s}\oplus C_s}(x)\downarrow \neq D_s(x)\).
    A contradiction.

    \((\Leftarrow)\) Suppose \(\Delta_e^{\sigma}(x)\downarrow = D_s(x)\) with \(C_s\in [\sigma]\).
    Let \((\theta,\sigma,x,k)\) be the diagonalizing axiom claimed in Item~\ref{it:Px_certified} at a stage \(t\) with \(s_0<t\le s\).
    Since \(A_0\res t\) is preserved and \(C_s\in [\sigma]\), we still have \(\Phi_{e,s}^{A_{0,s}\oplus C_s}(x)\downarrow = k = \Delta_e^{\sigma}(x)= D_s(x)\).

    \textit{(2):}
    Suppose towards a contradiction that \(\Phi_e^{A_0\oplus C}(x)\downarrow = D(x)\).
    There exists some stage \(s_1>s_0\) and quadruple \((\theta,\sigma,x,k)\) such that for each \(s\ge s_1\),
    this quadruple is the axiom of the computation \(\Phi_{e,s}^{A_{0,s}\oplus C_s}(x)\downarrow = k = D_s(x) = D(x)\).
    Therefore, we should perform the certification procedure for this computation at stage \(s_1\).
    Since \(C\in [\sigma]\), this computation is certified at \(s_1\).
    Therefore, \(\Delta_e^{C}(x)\downarrow=k\), contradicting the assumption that \(\Delta_e^{C}(x)\uparrow\).
\end{proof}

In other words, even before we run the \(P_e(x)\)-strategy, the following observations are valid.
If we observe that \(\Delta_e^{C_s}(x)\downarrow\neq D_s(x)\), then we know that \(\Phi_{e,s}^{A_{0,s}\oplus C_s}(x)\downarrow\neq D_s(x)\),
and therefore the \(P_e(x)\)-requirement is satisfied at stage \(s\).
If \(\Delta_e^{C_s}(x)\downarrow=D_s(x)\), then we know that Item~\ref{it:Px_defined} happens.

\textbf{\(P_e\)-strategy: }
At stage \(s\), start with \(x=0\);
\begin{enumerate}
    \item If \(x=s\), we stop;
    \item Perform the \(P_e(x)\)-strategy.
    If Item~\ref{it:Px_defined} or Item~\ref{it:Px_certified} happens
    (in the latter case, recall from \(P_e(x)\)-strategy that \(P_e(x)\) acts at stage \(s\) and we also say \(P_e\) \emph{acts} at stage \(s\)),
    we continue with \(x:=x+1\);
    otherwise, we define \(\tau_s(e)=x\) and stop.
\end{enumerate}
For each \(x<\tau_s(e)\), \(\Delta_e^{C_s}(x)\downarrow=D_s(x)\).

The following analogue of Lemma~\ref{lem:single_requirement} is a special case of Lemma~\ref{lem:Robinson_block_of_requirements}.
The proof given here serves as a warm-up; some of its arguments do not apply to the block case.
\begin{lemma}\label{lem:Robinson_single_requirement_warmup}
    Suppose the \(P_e\)-requirement is not initialized after \(s_0\).
    Then \(P_e\) is satisfied.
    Moreover, there exists some stage \(s_2>s_0\) after which \(P_e\) does not act.
    The restraint put on \(A_0\) is therefore bounded by \(s_2\).
\end{lemma}
\begin{proof}
    The set \(\mathcal{A} = \{x \mid \Delta_e^C(x)\downarrow \neq D(x)\}\) is \(\Sigma_1^C\)
    because \(x \in \mathcal{A}\) if and only if there exists a stage \(s\) and a string \(\sigma\)
    such that \(\Delta_{e,s}^\sigma(x) \neq D_s(x)\) and \(C \in [\sigma]\).
    Here, we use the fact that \(D\) is c.e.,
    and that whenever we define \(\Delta_e^\sigma(x)\), we define it correctly, which implies that \(D_s(x) = D(x)\) for such stage \(s\).

    If \(\mathcal{A}\) is nonempty, let \(a\) be any element in \(\mathcal{A}\) (or the least element in \(\mathcal{A}\) by \(\mathsf{L}\Sigma_1^C\)).
    If \(\mathcal{A}\) is empty, then \(\dom \Delta_e^C\neq M\) (otherwise, \(C\ge_T D\) yields a contradiction);
    we choose \(a\) to be the least element such that \(\Delta_e^C(a)\uparrow\), by \(\mathsf{L}\Pi_1^C\).

    For \(x<a\), we have \(\Delta_e^C(x)\downarrow\) (not necessarily \(=D(x)\)).
    Define the \(C\)-computable function \(w:a\to M\) by letting \(w(x)=s\) if \(\Delta_{e,s}^\sigma(x)\downarrow\) for some \(C\in [\sigma]\).
    Using \(\mathsf{I}\Sigma_1^C\) (in fact, \(\mathsf{B}\Sigma_1^C\) suffices), \(\ran w\) is bounded by some \(s_1>s_0\).
    That is, for each \(x<a\) and each \(s>s_1\), \(P_e(x)\) does not act at stage \(s\).

    \textit{Case (1):} \(\Delta_e^C(a)\downarrow\neq D(a)\).
    There exists some stage \(s_2\ge s_1\) and a string \(C\in [\sigma]\) such that \(\Delta_{e,s_2}^\sigma(a)\downarrow \neq D_{s_2}(a)=D(a)\).
    Then for all \(s>s_2\), we have \(\Delta_{e,s}^{C_s}(a)\downarrow \neq D_s(a)=D(a)\).
    By Lemma~\ref{lem:Px_lemma}(\ref{it:Px_lemma_1}), we have \(\Phi_{e,s}^{A_{0,s}\oplus C_s}(a)\downarrow \neq D_s(a)\)
    and hence \(\Phi_e^{A_0\oplus C}(a)\downarrow \neq D(a)\).
    Thus, for all \(s>s_2\), the \(P_e\)-requirement is satisfied.

    \textit{Case (2):} \(\Delta_e^C(a)\uparrow\).
    By Lemma~\ref{lem:Px_lemma}(\ref{it:Px_lemma_2}), \(P_e(a)\)-requirement is satisfied.
    Let \(W_j\) be the guessing set for \(P_e(a)\)-strategy.
    Since \(F_e(x)\) is bounded, there exists some stage \(s_2>s_1\) such that for each \(s>s_2\), no computation is certified and also \(\Delta_e^{C_s}(a)\uparrow\).
    Consequently, \(P_e(a)\) does not act at stage \(s>s_2\).
\end{proof}

We remark that the argument used in \textit{Case (2)} of the preceding proof does not extend to the block case (Lemma~\ref{lem:Robinson_block_of_requirements}).
The reason is that the finiteness behavior of \(F_e(x)\) is no longer controllable.

\subsection{Block of requirements}\label{sec:Robinson_block_of_requirements}
Let \(\Lambda\) be a block of \(P\)-requirements.
Assume \(\Lambda=\{P_{e_0}, P_{e_0+1}, \ldots, P_{e_1}\}\).
To simplify notation, we sometimes identify \(\Lambda\) with the index set \(\{e_0, e_0+1, \ldots, e_1\}\).
Each of the requirements in \(\Lambda\) follows the \(P_e\)-strategy, and they have no conflicts with each other.
Now we shall prove the analogue of Lemma~\ref{lem:block_of_requirements}:
\begin{lemma}\label{lem:Robinson_block_of_requirements}
Suppose there exists some \(s_0\) after which \(\Lambda\) is not initialized.
Then all requirements in \(\Lambda\) are satisfied, and there exists some \(s_3 > s_0\) after which no requirement in \(\Lambda\) acts.
Moreover, the restraint put on \(A_0\) is bounded by \(s_3\).
\end{lemma}
\begin{proof}
    We define \(f:\Lambda\to M\) by
    \[
        f(e)=s \text{ if } s>s_0 \land \exists x \exists \sigma \left(\Delta_{e,s}^\sigma(x)\downarrow \neq D_s(x) \land C\in [\sigma]\right),
    \]
    where the \(x\) in the statement is not necessarily the least one. For such \(x\) and \(s\), we also have \(D_s(x)=D(x)\) as \(D\) is a c.e.\ set.
    Since \(f\) is a partial \(C\)-computable function, \(\ran f\) is bounded by some \(s_1\).
    That is, for each \(e\in \dom f\) and each stage \(s>s_1\), \(P_e\) does not act at stage \(s\).
    As \(\dom f\) is a \(\Sigma_1^C\) subset of \(\Lambda\), \(\dom f\) is \(M\)-finite.
    We let \(\Lambda_1 = \Lambda \setminus\dom f\).

    Then we claim that \(\bigcup_{e\in \Lambda_1} \dom \Delta_e^C \neq M\).
    To see this, suppose otherwise.
    Then for any \(x\in M\), we can search for an \(e\in \Lambda_1\), a string \(\sigma\), and a stage \(s>s_1\)
    such that \(\Delta_e^{\sigma}(x)=D_s(x)\) is defined at stage \(s\) and \(C\in[\sigma]\), and output \(D(x)=D_s(x)\).
    This procedure shows that \(D\le_T C\), contradicting the assumption that \(D\nleq_T C\).
    Therefore, \(\bigcup_{e\in \Lambda_1} \dom \Delta_e^C\) is bounded by some \(b\).

    Define \(g:\Lambda_1\times b\to M\) by
    \[
        g(e,x)=s \text{ if } s>s_1\land \exists \sigma \left(\Delta_{e,s}^\sigma(x)\downarrow \land C\in [\sigma]\right).
    \]
    So \(\ran g\) is bounded by some \(s_2\).
    This means that for each \(e\in \Lambda_1\) and for each \(x\in \dom \Delta_e^C\), \(P_e(x)\) does not act at stage \(s>s_2\).

    Next we attempt to locate, for each \(e\in \Lambda_1\), the least element \(a_e\) such that \(\Delta_e^C(a_e)\uparrow\).
    \(P_e(a_e)\) may still act after \(s_2\).

    We define the partial computable function \(a:\Lambda_1\to M\) by
    \[
        a(e)=x \text{ if we observe that \(P_e(x)\) acts at some (least) stage \(s>s_2\)}.
    \]
    Note that for such \(x\), we must have \(\Delta_e^C(x)\uparrow\), for otherwise this contradicts the choice of \(s_2\).
    Note that for \(e\in \Lambda_1\setminus \dom a\), \(P_e\) does not act at stage \(s>s_2\).
    Let \(\Lambda_2 = \dom a\).
    Then \(\Lambda_2\) is \(M\)-finite, and \(a\) is total computable on \(\Lambda_2\).
    We let \(a_e=a(e)\).

    For each \(e\in \Lambda_2\), we let \(W_{j_e}\) be the guessing set for the \(P_e(a_e)\)-strategy.
    The index \(j_e\) can be computed directly for \(e\in \Lambda_2\).
    Since \(\Delta_e^C(a_e)\uparrow\), we have \(\lim_s p(j_e,s) = 0\).
    Let \(s_e\) be such that for each \(s>s_e\), \(p(j_e,s) = 0\).
    This means that for each \(s>s_e\), no computation is certified and therefore \(P_e(a_e)\) does not act at stage \(s>s_e\).
    Now we shall find a bound for \(\{s_e\mid e\in \Lambda_2\}\).

    Note that \(q:\{j_e \mid e\in \Lambda_2\}\to M\) is computable (recall that \(q(j_e)\) bounds the number of \(\{s\mid p(j_e,s+1)\neq p(j_e,s)\}\)).
    Therefore, \(\ran q\) is bounded by some \(d\).
    Define
    \(w:\Lambda_2\times d\to M\) by
    \[
    w(e,k-1)=s \text{ if } p(j_e,s+1)\neq p(j_e,s)\text{ for the }k\text{-th time.}
    \]
    Then \(\ran w\) is bounded by some \(s_3\).
    That is, for each \(s>s_3\), \(p(j_e,s)=0\), and \(s_3\) is a bound for \(\{s_e\mid e\in \Lambda_2\}\).

    Putting this together, we have that for each \(e\in \Lambda\), \(P_e\) does not act at any stage \(s>s_3\).
\end{proof}

Dynamic priority assignments are the same as in Section~\ref{sec:SST_dynamic_assignments}.
So we move directly to the construction and verification, which is analogous to Section~\ref{sec:SST_construction}.

\subsection{Construction and verification}\label{sec:Robinson_construction_and_verification}
Most parts remain exactly the same.
Here, we only rephrase Part II.

\textbf{Part II. Block strategy:}
Let the block \(\Lambda(0)\) be eligible to act.
Suppose, without loss of generality, that the block \(\Lambda(i)\) is eligible to act.
For each \(P_e\in \Lambda(i)\), we perform the \(P_e\)-strategy, presented in Section~\ref{sec:Robinson_single_requirement}.
If some requirement acts, we initialize lower-priority blocks and start Part III;
otherwise, we let the block with the next highest priority be eligible to act
until we reach a block that contains the \(P_s\)-requirement or the \(Q_s\)-requirement,
at which point we start Part III.

This completes the construction.

Lemma~\ref{lem:block_not_initialized} and Lemma~\ref{lem:SST_satisfied} still hold,
with Lemma~\ref{lem:block_of_requirements} replaced by Lemma~\ref{lem:Robinson_block_of_requirements}.
This completes the verification.

\section{Discussion}\label{sec:discussion}
We point out that the original Robinson Splitting Theorem can be proved in models of \(\mathsf{P}^-+\mathsf{B}\Sigma_2\). This is not surprising, since \(\mathsf{B}\Sigma_2\) even implies Sacks Density Theorem~\cite{RN320}.
Recall that~\cite{RN365,RN320} shows that, under \(\mathsf{B}\Sigma_2\), every incomplete c.e.\ set is hyperregular.
Therefore, most of our proof applies to an arbitrary low set \(C\).
The only additional issue is that the proof of Lemma~\ref{lem:Robinson_block_of_requirements} in our current form also uses that \(X\) is \(\omega\)-c.e.\@, where \(X\) is the set for Robinson's trick.
Under \(\mathsf{B}\Sigma_2\), however, this lemma has a much shorter proof:
define \(f:\Lambda\to M\) by
\[
    f(e)=s \text{ if } \forall t>s\left(P_e \text{ does not act at stage }t\right)
\]
This is a \(\Pi_1\)-definable function, and \(\mathsf{B}\Pi_1\) (equivalent to \(\mathsf{B}\Sigma_2\)) implies that \(\ran f\) is bounded.
This yields Lemma~\ref{lem:Robinson_block_of_requirements}.
Hence we obtain the following theorem:
\begin{theorem}
    Robinson Splitting Theorem (Theorem~\ref{thm:robinson}) holds in models of \(\mathsf{P}^-+\mathsf{B}\Sigma_2\).
\end{theorem}

We leave the following question open:
\begin{question}\label{ques:superlow}
Does \(\mathsf{I}\Sigma_1\) prove the original Robinson Splitting Theorem when the set \(C\) is low? Or is it equivalent to \(\mathsf{B}\Sigma_2\) over \(\mathsf{P}^-+\mathsf{I}\Sigma_1\)?
\end{question}
Our argument relies heavily on the fact that \(X\) and \(C'\) are \(\omega\)-c.e.\@.
For a general low set \(C\), the jump \(C'\) may be \(\alpha\)-c.e.\ for some computable ordinal \(\alpha>\omega\),
which could create additional difficulties for constructions formalized in \(\mathsf{I}\Sigma_1\).
This suggests that one must treat lowness more carefully over \(\mathsf{P}^-+\mathsf{I}\Sigma_1\).
On the other hand, it would still be surprising if a finite-injury argument failed in \(\mathsf{I}\Sigma_1\).

This question may also be related to the following special case of the Sacks Density Theorem:
\begin{question}
    Does \(\mathsf{I}\Sigma_1\) prove the upper density property:
    For every c.e.\ degree \(C<_T \varnothing'\), there exists a c.e.\ degree \(D\) such that \(C<_T D<_T \varnothing'\)?
\end{question}
It would be unexpected if the upper density property already failed for low degrees \(C\) in some model of \(\mathsf{I}\Sigma_1\).

\bibliographystyle{plain}
\bibliography{ref.bib}
\end{document}